\documentclass{article}
\usepackage{verbatim}
\usepackage{graphicx}
\usepackage{listings}
\usepackage{amsmath,amssymb,amsthm}
\usepackage[utf8]{inputenc}
\usepackage[english]{babel}
\usepackage[top=4cm, bottom=4cm, left=2.5cm, right=2.5cm]{geometry}
\usepackage{fullpage}
\usepackage{colonequals}
\usepackage[colorlinks=true,linkcolor=black,urlcolor= black,citecolor=black]{hyperref}
\usepackage{graphicx}
\usepackage{color}
\usepackage{indentfirst}
\usepackage[OT2,T1]{fontenc}
\usepackage{comment} 
\usepackage{shadethm}
\usepackage{enumerate}
\usepackage{soul}
\usepackage{mathtools}
\usepackage{tikz}
\usetikzlibrary{matrix}
\usepackage[nottoc, notlof, notlot]{tocbibind}
\newtheorem{mydef}{Definition}[section]
\newtheorem{mypp}[mydef]{Proposition}
\newtheorem{myppr}[mydef]{Property}
\newtheorem{mythm}[mydef]{Theorem}
\newtheorem{mylem}[mydef]{Lemma}

\newtheorem{myrem}[mydef]{Remark}

\newtheorem{mynot}[mydef]{Notation}

\usepackage{color}

\DeclareMathOperator{\Id}{Id}
\usepackage{microtype}

\title{
	Local well-posedness result for a class of non-local quasi-linear systems and its
	application to the justification of Whitham-Boussinesq systems}
\author{Louis Emerald}

\begin{document}
\numberwithin{equation}{section}
\maketitle

\begin{abstract}
In this paper we prove a local well-posedness result for a class of quasi-linear systems of hyperbolic type involving Fourier multipliers. Among the physically relevant systems in this class is a family of Whitham-Boussinesq systems arising in the modeling free-surface water waves. Our result allows to prove the rigorous justification of these systems as approximations to the general water waves system on a relevant time scale, independent of the shallowness parameter.       
\end{abstract}

\section{Introduction}
\subsection{Motivations}
In this work we prove a well-posedness result and a stability result for systems of the form
\begin{align}\label{general Whitham-Boussinesq equation introduction}
\begin{cases}
\partial_t \zeta + (\mathrm{G}_1)^2 \nabla\cdot v + \epsilon \mathrm{G}_2\nabla\cdot(\zeta \mathrm{G}_2 [v]) = 0,\\
\partial_t v + \nabla \zeta + \epsilon(\mathrm{G}_2[v] \cdot \nabla) \mathrm{G}_2[v] = 0,
\end{cases}
\end{align}
where $(\mathrm{G}_1,\mathrm{G}_2)$ are {\em admissible Fourier multipliers} (see definition below), and the unknowns $\zeta$ and $v$ are functions of time $t$ and space $x \in \mathbb{R}^d$ ($d\in\{1,2\}$) valued in $\mathbb{R}$ and $\mathbb{R}^d$ respectively, and $\epsilon\ge 0$.

Several standard equations fit into the above framework. For instance, setting $\mathrm{G}_1=\mathrm{G}_2=\Id$ one obtains the standard shallow-water system, which can be viewed as a special case of the $d$-dimensional isentropic, compressible Euler equation for ideal gases with quadratic pressure law, and whose well-posedness theory is quite standard (see {\em e.g.} \cite{Metivier2008}). This work extends this theory to the general class of equations above. 
	Now, if we set $\mathrm{G}_2=(\Id- \mu b\Delta_x)^{-1}$ and $\mathrm{G}_1=(\Id+ \mu a\Delta_x)^{1/2}(\Id- \mu b \Delta_x)^{-1}$ with $a,b,\mu\ge0$, we obtain (setting $u=\mathrm{G}_2[v]$) the so-called $abcd$-Boussinesq system with $c=0$ and $d=b$. These systems model surface gravity waves in the long wave regime, and have been introduced in full generality by Bona, Chen and Saut in \cite{BonaChenSaut02}. Concerning the well-posedness of the initial-value problem, some early results were proved by the energy method and using the regularization properties of the operator $(\Id+ \Delta_x)^{-1/2}$ in \cite{BonaChenSaut04,Anh10}, but the dependency of the result (and in particular the existence time) with respect to $\epsilon$ and $\mu$ were not specified. Following \cite{DougalisMitsotakisSaut07}, Linares, Pilod and Saut \cite{LinaresPilodSaut12} used dispersive techniques to prove the well-posedness of a large class of Boussinesq systems (when $d=2$) on a time-interval of length $T\gtrsim \epsilon^{-1/2}$ when $\mu\approx\epsilon$. Subsequently, Saut and Xu \cite{SautXu12} (see also \cite{MingSautZhang12,Burtea16,SautXu21}) improved this result by proving well-posedness results on a time interval of length $T\gtrsim \epsilon^{-1}$ by using symmetrization and energy techniques. Importantly, the latter result holds for $d\in\{1,2\}$ and on the full shallow-water regime, that is  $\mu\in(0,1]$, as the authors do not make use of the dispersive nature of the equations. This property is important as it allows to rigorously justify Boussinesq systems as an asymptotic model, improving the shallow-water system (which corresponds to setting $\mu=0$) in the long wave regime, $\epsilon\lesssim\mu$; see the memoir of Lannes \cite{Lannes}.

Our main interest in considering \eqref{general Whitham-Boussinesq equation introduction} stems from the so-called Whitham-Boussinesq systems. Indeed, setting $\mathrm{G}_1^2=\mathrm{G}_2=\frac{\tanh(\sqrt\mu|D|)}{\sqrt\mu|D|}$, we obtain the system introduced by Dinvay, Dutykh and Kalisch in~\cite{DinvayDutykhKalisch19}. It has been recently proved by the author \cite{Emerald2020} that Whitham-Boussinesq systems are approximations of order $O(\mu\epsilon)$, as opposed to $O(\mu)$ for the shallow-water system and $O(\mu^2+\mu\epsilon)$ for Boussinesq systems, of the general water waves model, in the sense of consistency (see Proposition \ref{consistency} below for a precise statement). 
The consistency property is the first step to prove the full justification of a model for surface gravity waves. The second step is to prove stability estimates and the local well-posedness of the model for sufficently regular data on the relevant timescale. The well-posedness of the aforementioned Whitham-Boussinesq systems has been studied in \cite{Dinvay19, DinvaySelbergTesfahun19, Tesfahun22, DenekeDuferaTesfahun22}. 
In \cite{Dinvay19}, Dinvay diagonalizes linear terms and uses strongly the regularization properties of the operator $(\mathrm{G}_2)^{-1}$, from which stems an existence time of length $T\gtrsim \epsilon \mu^{-1/2}$ (in dimension $d=1$).
In \cite{DinvaySelbergTesfahun19}, Dinvay, Selberg and Tesfahun exploit the dispersive nature of the system. This allows them to prove the well-posedness at the energy level in dimension $d=1$, from which global-in-time well-posedness (for sufficiently small initial data) follows. They also prove the local-in-time existence in dimension $d=2$. Yet, due to the use of dispersive estimates, these two results do not provide the control of solutions and their derivatives uniformly with respect to $\mu\in(0,1]$, as required for the rigorous justification of the system as an asymptotic model for water waves. The precise dependency of the time on which solutions exist and are controlled (in a ball of twice the size of the initial data in the relevant Banach space) is clarified in subsequent works: the result of Tesfahun \cite{Tesfahun22} in dimension $d=2$ exhibits a time interval of length $T\gtrsim \epsilon^{-2+\delta} \mu^{3/2-\delta}$ with $\delta>0$ arbitrarily small, while the result of Deneke, Dufera and Tesfahun \cite{DenekeDuferaTesfahun22} in dimension $d=1$ exhibits a time interval of length $T\gtrsim \epsilon^{-1}\mu^{1/2}$.

Let us conclude this state of the art by mentioning the recent work of Paulsen \cite{Paulsen22} where the well-posedness and control of some Whitham-Boussinesq systems, different from the one mentioned above and considered here, is proved on a time-interval of length $T\gtrsim \epsilon^{-1}$ for parameters on the shallow-water regime $(\epsilon,\mu)\in (0,1]^2$. The strategy of the proof is, similarly to ours, based on the energy method and, as stated therein, the two works complete each other well.

\subsection{Definitions and notations}

\begin{mydef}
    For $u:\mathbb{R}^d \to \mathbb{R}^d$ a tempered distribution, denote $\widehat{u}$ its Fourier transform. Let $G:\mathbb{R}^d\to\mathbb{R}$ be a bounded function. The Fourier multiplier associated with $G(\xi)$ is denoted $\mathrm{G}\colonequals G(D)$ and defined by 
    \begin{align*}
   \forall u\in L^2(\mathbb{R}^d),\quad      \widehat{\mathrm{G}[u]}(\xi) = G(\xi)\widehat{u}(\xi).
    \end{align*}
\end{mydef}

\begin{mydef}
    We say that a couple of Fourier multipliers $(\mathrm{G}_1,\mathrm{G}_2)$ is admissible if its symbols satisfy
    \begin{itemize}
        \item for $k\in\{1,2\}$, $G_k\in L^{\infty}(\mathbb{R}^d)$ and $\langle \cdot\rangle \nabla G_k \in L^\infty(\mathbb{R}^d)^d$;
        \item for all $\xi \in \mathbb{R}^d$, we have $G_1(\xi)>0$;
        \item for all $\xi \in \mathbb{R}^d$, we have $|G_2(\xi)| \leq G_1(\xi)$.
    \end{itemize}
\end{mydef}

Denoting $U \colonequals \begin{pmatrix} \zeta \\ v \end{pmatrix}$, we can write systems \eqref{general Whitham-Boussinesq equation introduction} under their matricial form
\begin{align}\label{WB matricial form introduction}
    \partial_t U + \sum\limits_{j=1}^d A_j(U)[\partial_j U] = 0,
\end{align}
where (in dimension $d=2$, the analogous definitions when $d=1$ is straightforward)
\begin{align}\label{A_j introduction}
    \begin{cases} A_1(U)[\circ] = \begin{pmatrix} \epsilon \mathrm{G}_2 [ \mathrm{G}_2[v_1]\circ] & (\mathrm{G}_1)^2[\circ] + \epsilon \mathrm{G}_2[\zeta \mathrm{G}_2[\circ]] & 0 \\ 1 & \epsilon \mathrm{G}_2[v_1] \mathrm{G}_2[\circ] & 0 \\ 0 & 0 & \epsilon \mathrm{G}_2[v_1] \mathrm{G}_2[\circ] \end{pmatrix}, \\ 
    A_2(U)[\circ] = \begin{pmatrix}  \epsilon \mathrm{G}_2 [\mathrm{G}_2[v_2]\circ] & 0 & (\mathrm{G}_1)^2[\circ] + \epsilon\mathrm{G_2}[\zeta\mathrm{G}_2[\circ]] \\ 0 & \epsilon\mathrm{G}_2[v_2] \mathrm{G}_2[\circ] & 0 \\ 1 & 0 & \epsilon\mathrm{G}_2[v_2] \mathrm{G}_2[\circ] \end{pmatrix}.
    \end{cases}
\end{align}

The natural functional setting is given by the energy norms.
\begin{mydef}
    \begin{itemize}
        \item We denote by $\mathcal{S}'(\mathbb{R}^d)$ the set of tempered distributions.
        \item We denote by respectively $|\cdot|_2$ and $(\cdot,\cdot)_2$, the norm and the scalar product in $L^2(\mathbb{R}^d)$.
        
        \item Let $s\geq 0$. We denote by $H^s(\mathbb{R}^d)$ the Sobolev spaces of order $s$ in $L^2(\mathbb{R}^d)$. Denoting $\Lambda^s \colonequals (1-\Delta)^{s/2}$, where $\Delta$ is the Laplace operator in $\mathbb{R}^d$, the norm associated with $H^s(\mathbb{R}^d)$ is $|\cdot|_{H^s} \colonequals |\Lambda^{s}\cdot|_2$.
        
        \item Let $\mathrm{G}_1$ be a Fourier multiplier of order $0$, defined by positive function $G_1$. Let also $s \geq 0$. We define the Banach spaces $X^s(\mathbb{R}^d)$ and $Y^s(\mathbb{R}^d)$ by 
        \begin{align*}
            &X^s(\mathbb{R}^d) \colonequals \big\{ U = \big(\,\zeta\,,v\,\big) \in \mathcal{S}'(\mathbb{R}^d)\times \mathcal{S}'(\mathbb{R}^d)^d, \ \ |U|_{X^s} < +\infty \big\}, \\
            &Y^s(\mathbb{R}^d) \colonequals \big\{ U = \big(\,\zeta\,,v\,\big) \in \mathcal{S}'(\mathbb{R}^d)\times \mathcal{S}'(\mathbb{R}^d)^d, \ \ |U|_{Y^s} < +\infty \big\},
        \end{align*}
        where $|U|_{X^s} \colonequals |\zeta|_{H^s} + |\mathrm{G}_1[v]|_{H^s}$ and $|U|_{Y^s} \colonequals |\zeta|_{H^s} + |\mathrm{G}_1^{-1}[v]|_{H^s}$. These are the energy norms associated with system \eqref{WB matricial form introduction}.
    \end{itemize}
    Henceforth, we denote $C(\lambda_1,\lambda_2,\dots)$ a positive constant depending non-decreasingly on its parameters. Unless it is essential, the dependency with respect to the regularity index $s$ is omitted.
    
    We write $a\lesssim b$ for $a\leq Cb$ with $C>0$ a positive number whose dependency is unessential or clear from the context. We denote $a\approx b$ when $a\lesssim b$ and $b\lesssim a$. We denote  $\langle \cdot \rangle \coloneqq (1+|\cdot|^2)^{1/2}$.
    
\end{mydef}

\subsection{Main results}
The systems \eqref{WB matricial form introduction} are symmetrizable quasi-linear hyperbolic. The main key to prove the local well-posedness of such systems is the following energy estimates on the linearized system.
\begin{mypp}\label{energy estimates introduction}
    Let $s \geq 0$, $t_0 > d/2$ and $(\mathrm{G}_1,\mathrm{G}_2)$ be admissible Fourier multipliers. For any $\epsilon\in(0,1]$, $T>0$, $\underline{U} = \big( \underline{\zeta} , \underline{v} \big) \in W^{1,\infty}([0,T/\epsilon],X^{t_0}(\mathbb{R}^d))\cap L^{\infty}([0,T/\epsilon],X^{\max{(t_0+1,s)}}(\mathbb{R}^d))$ for which there exists $h_{\min} > 0$ such that for all $(t,x) \in [0,T/\epsilon]\times \mathbb{R}^d$ 
    \begin{align}\label{non cavitation}
        1 + \epsilon \underline{\zeta} \geq h_{\min},
    \end{align} 
    and for any $U = \big( \zeta , v \big) \in W^{1,\infty}([0,T/\epsilon],X^{s}(\mathbb{R}^d))\cap L^{\infty}([0,T/\epsilon],X^{s+1}(\mathbb{R}^d))$ satisfying the system
    \begin{align*}
        \partial_t U + \sum\limits_{j=1}^d A_j(\underline{U})[\partial_j U] = \epsilon R,
    \end{align*}
    where $R \in L^{\infty}([0,T/\epsilon],X^s(\mathbb{R}^d))$, and for $j = 1,2$, $A_j(\underline{U})$ is defined by \eqref{A_j introduction}, we have for any $t\in [0,T/\epsilon]$,
    \begin{align}\label{energy estimates of order s introduction}
        |U|_{X^s} \leq \kappa_0 e^{\epsilon \lambda_s t} |U|_{X^s}|_{t=0} + \epsilon \nu_s \int_{0}^{t} |R(t')|_{X^s} \rm{dt}',
    \end{align}
    where $\lambda_s, \nu_s := C(\frac{1}{h_{\min}},T,|\underline{U}|_{W^{1,\infty}_t X^{t_0}},|\underline{U}|_{L^{\infty}_t X^{\max{(t_0+1,s)}}})$ and $\kappa_0 \colonequals C(\frac{1}{h_{\min}},|\underline{U}|_{X^{t_0}}|_{t=0})$.
\end{mypp}

Using a Picard iteration scheme and the regularization method from Chapter 7 in \cite{Metivier2008} we infer the following well-posedness result on the systems \eqref{WB matricial form introduction}.
\begin{mythm}\label{Thm wanted}
    Let $s > d/2 +1$, $h_{\min} >0$ and $M > 0$. Let also $(\mathrm{G}_1,\mathrm{G}_2)$ be a couple of admissible Fourier multipliers. There exist $T > 0$ and $C > 0$ such that for all $\epsilon \in (0,1]$, $U_0 \in X^s(\mathbb{R}^d)$ with $|U_0|_{X^s} \leq M$ and satisfying \eqref{non cavitation}, there exists a unique solution $U \in C^0([0,T/\epsilon],X^s(\mathbb{R}^d))$ of the Cauchy problem
    \begin{align*}
        \begin{cases}
            \partial_t U + \sum\limits_{j=1}^d A_j(U)\partial_j U = 0, \\
            U|_{t=0} = U_0.
        \end{cases}
    \end{align*}
    Moreover $|U|_{L^{\infty}([0,T/\epsilon], X^s)} \leq C |U_0|_{X^s}$.
\end{mythm}

We also have the following stability result.
\begin{mypp}\label{Stability}
    Let the assumptions of Theorem \ref{Thm wanted} be satisfied and use the notations therein. Assume also that there exists $\widetilde{U} \in C([0,\widetilde T/\epsilon],X^s(\mathbb{R}^d))$ solution of 
    \begin{align*}
        \partial_t\widetilde{U} + \sum\limits_{j=1}^d A_j(\widetilde{U})\partial_j\widetilde{U} = \widetilde{R},
    \end{align*}
    where $\widetilde{R} \in L^{\infty}([0,T/\epsilon],X^{s-1}(\mathbb{R}^d))$. Then, the error with respect to the solution $U \in C^0([0,T/\epsilon],X^s(\mathbb{R}^d))$ given by Theorem \ref{Thm wanted} satisfies for all times $t \in [0,\min{(\widetilde T,T)}/\epsilon]$,
    \begin{align*}
        |\mathfrak{e}|_{L^{\infty}([0,t],X^{s-1})} \leq C(\tfrac{1}{h_{\min}},|U|_{L^{\infty}([0,t], X^s)},|\widetilde{U}|_{L^{\infty}([0,t], X^s)})(|\mathfrak{e}|_{X^{s-1}}|_{t=0} + t|\widetilde{R}|_{L^{\infty}([0,t], X^{s-1})}),
    \end{align*}
    where $\mathfrak{e} \coloneqq U - \widetilde{U}$.
\end{mypp}

\begin{myrem}\label{rem}
	In all previous statement, the dependency of the constants and the existence time with respect to the admissible pair of Fourier multipliers $(\mathrm{G}_1,\mathrm{G}_2)$ occurs only through $|(G_k,\langle \cdot \rangle \nabla G_k)|_{L^\infty}$ for $k\in\{1,2\}$.
\end{myrem}

Let us now turn to the rigorous justification of the systems \eqref{WB matricial form introduction} in the context of irrotational free surface flows. We first recall the notations and the physical meaning of the different variables (see \cite{Emerald2020}).
\begin{itemize}
    \item The free surface elevation is the graph of $\zeta$, which is a function of time $t$ and horizontal space $x\in\mathbb{R}^d$.
    \item $v(t,x)$ is the gradient of the trace at the surface of the velocity potential.
\end{itemize}
Moreover every variable and function in \eqref{WB matricial form introduction} is compared with physical characteristic parameters of the same dimension. Among those are the characteristic water depth $H_0$, the characteristic wave amplitude $a_{\rm{surf}}$ and the characteristic wavelength  $L$.
These physical characteristic parameters define two dimensionless paramaters of main importance:
    \begin{align*}
        \mu \colonequals \frac{H_0^{2}}{L^{2}}, \ \ \rm{and} \ \ \epsilon  \colonequals \frac{a_{\rm{surf}}}{H_0}.
    \end{align*}
    The first parameter, $\mu$, is called the shallow water parameter. The second parameter, $\epsilon$, is called the nonlinearity parameter. In the following we restrict ourselves to the shallow-water regime: $(\mu,\epsilon) \in (0,1]^2$.

\begin{mynot}\label{notation G dependant de mu}
    From now on the Fourier multipliers are denoted $\mathrm{G}_1^\mu$ and $\mathrm{G}_2^\mu$ as they depend on $\mu\in(0,1]$ and are of the form $\mathrm{G}_k^{\mu} = G_k(\sqrt{\mu}D)$, $k = 1,2$ where the symbols $G_1$ and $G_2$ are independent of $\mu$.
\end{mynot} 

Proposition 1.15 in \cite{Emerald2020} can be easily extended to the following result.
\begin{mypp}\label{consistency}
    Let  $s \geq 0$. In \eqref{WB matricial form introduction}, let $\mathrm{G}_1^{\mu} \colonequals \sqrt{\frac{\tanh{(\sqrt{\mu}|D|)}}{\sqrt{\mu}|D|}}$. Let also $\mathrm{G}_2^{\mu}$ be a Fourier multiplier such that $|G_2^{\mu}(\xi)-1|\lesssim \mu |\xi|^2$.
    Then any classical solution $(\zeta,\psi)$ of the water waves equations satisfying the non-cavitation hypothesis \eqref{non cavitation}, with $U = (\zeta,\nabla\psi) \in C^0([0,\widetilde T/\epsilon],H^{s+4}(\mathbb{R}^d)^{1+d})$, satisfy the system \eqref{WB matricial form introduction} up to a remainder term of order $O(\mu\epsilon)$, i.e. for any $t\in[0,\widetilde T/\epsilon]$,
    \begin{align*}
        \partial_t U + \sum\limits_{j=1}^d A_j(U)\partial_j U = \mu\epsilon R,
    \end{align*}
    where $|R(t,\cdot)|_{H^s} \leq C(\frac{1}{h_{\min}},|\zeta|_{H^{s+4}},|\nabla\psi|_{H^{s+4}})$, uniformly with respect to $(\mu,\epsilon)\in (0,1]^2$. 
    
    \indent 
    
    We say that the water waves equations are consistent at order $O(\mu\epsilon)$ with the systems \eqref{WB matricial form introduction} in the shallow water regime. 
\end{mypp}

\begin{myrem}
Within the assumptions of Proposition \ref{consistency}, the systems \eqref{WB matricial form introduction} are Whitham-Boussinesq systems (see \cite{Emerald2020}).
As aforementioned, setting $\mathrm{G}_2^{\mu} = (\mathrm{G}_1^{\mu})^2$, we obtain the system by Dinvay, Dutykh and Kalisch in~\cite{DinvayDutykhKalisch19}. Previously existing well-posedness results use the regularizing effect of $(\mathrm{G}_1^{\mu})^2$ which gives the system a semi-linear structure (namely the system can be solved through the Duhamel formula).
Theorem \ref{Thm wanted} allows us to set, for instance, $\mathrm{G}_2^{\mu} = \mathrm{G}_1^{\mu}$, thus proving the local well-posedness for a Whitham-Boussinesq system with a genuinely quasi-linear structure. 
When $\mathrm{G}_2^{\mu} = \Id$, the system \eqref{WB matricial form introduction} is a Whitham-Boussinesq system, yet the question of local well-posedness for this system is open.
\end{myrem}

From Theorem \ref{Thm wanted}, Proposition \ref{Stability} and Proposition \ref{consistency} we infer the full justification of a class of Whitham-Boussinesq systems.

\begin{mythm}\label{justification}
	Under the assumption and using the notation of Proposition \ref{consistency}, and provided  $(\mathrm{G}_1,\mathrm{G}_2)$ are admissible Fourier multipliers, then for any $U = (\zeta,\nabla\psi) \in C^0([0,\widetilde T/\epsilon],H^{s+4}(\mathbb{R}^d))$ classical solution of the water waves equations and satisfying the non-cavitation assumption \eqref{non cavitation}, there exists a unique $U_{\mathrm{WB}} = (\zeta_{\mathrm{WB}},v_{\mathrm{WB}}) \in C^0([0,T/\epsilon],X^{s+4}(\mathbb{R}^d))$ classical solution of the Whitham-Boussinesq systems \eqref{WB matricial form introduction} with initial data $U_{\mathrm{WB}}|_{t=0} = (\zeta|_{t=0},\nabla\psi|_{t=0})$, and one has for all times $t \in [0,\min{(\widetilde T,T)}/\epsilon]$,
   \begin{align*}
       |U - U_{\mathrm{WB}}|_{L^{\infty}([0,t],X^{s})} \leq C\, \mu\epsilon t ,
    \end{align*}
    with $T$ (provided by Theorem \ref{Thm wanted}) and  $C=C(\tfrac{1}{h_{\min}},|U|_{L^{\infty}([0,T/\epsilon], H^{s+4})})$ uniform with respect to $(\mu,\epsilon)\in (0,1]^2$.
\end{mythm}

\begin{myrem}
    Regular solutions of the water waves equations as in Proposition \ref{consistency} and Theorem \ref{justification} are provided by Theorem 4.16 in \cite{WWP}.
\end{myrem}

\subsection{Outline}
Section \ref{Energy estimates section} is dedicated to the proof of Proposition \ref{energy estimates introduction}. In Subsection \ref{Symmetrization of the systems section} we focus on the symmetrization of the systems \eqref{general Whitham-Boussinesq equation introduction}. In Subsection \ref{Energy estimates of order 0 section}, we prove the energy estimates of Proposition \ref{energy estimates introduction} in the case of $s=0$. Finally in Subsection \ref{Energy estimates of higher order section}, we prove the general case $s \geq 0$.

Section \ref{Local well-posedness and stability section} is dedicated to the proof of Theorem \ref{Thm wanted}. In Subsection \ref{Well-posedness of the linearized systems section} we prove a local well-posedness result for the systems \eqref{general Whitham-Boussinesq equation introduction} linearized around a sufficiently regular state. In Subsection \ref{Well-posedness of the systems section} we focus on the proof of Theorem \ref{Thm wanted}. In Subsection \ref{Blow up criterion section} we establish a blow-up criterion for the local well-posedness of the systems \eqref{general Whitham-Boussinesq equation introduction}.

In Section \ref{Stability result section} we prove Proposition \ref{Stability}.

In Section \ref{Full justification of a class of Whitham-Boussinesq systems section} we prove Theorem \ref{justification}.

Appendix \ref{annex} collects useful technical results.

\section{Energy estimates}
\label{Energy estimates section}
\subsection{Symmetrization}
\label{Symmetrization of the systems section}
In this subsection we focus on the symmetrization of the systems \eqref{WB matricial form introduction}. We perform the computations in the setting $d=2$, the case $d=1$ is obtained in the same way.

\begin{myppr}\label{symmetrization of the WB system}
    Let $s \geq 0$ and $t_0 > d/2$. For any $U \in X^{t_0+1}(\mathbb{R}^d)$, let
    \begin{align}\label{symmetrizer}
        S_0(U)[\circ] \colonequals \begin{pmatrix} 1 & 0 \\ 0 & (\mathrm{G}_1)^2[\circ] + \epsilon\mathrm{G}_2[\zeta \mathrm{G}_2[\circ]] \end{pmatrix},
    \end{align}
    be an operator defined in $X^0(\mathbb{R}^d)$. Applying the latter operator to system \eqref{WB matricial form introduction} we get
    \begin{align}\label{symmetrized Wb equation general form}
        S_0(U)[\partial_t U] + \sum\limits_{j=1}^d \widetilde{A}_j(U)[\partial_j U] = \epsilon \sum\limits_{j=1}^d F_j(U)[\partial_j U],
    \end{align}
    where for $j = 1,2 $, $\widetilde{A}_j(U)$ is a symmetric operator defined by ( denoting by * the adjoint in $L^2(\mathbb{R}^d)$)
    \begin{align*}
        \widetilde{A}_j(U) = \frac{B_j(U)+B_j(U)^*}{2},
    \end{align*}
    with 
    \begin{align*}
        B_1(U)[\circ] = \begin{pmatrix} B_{1}^{1,1}(U)[\circ] & B_1^{1,2}(U)[\circ] & 0 \\ B_1^{1,2}(U)[\circ] & B_1^{2,2}(U)[\circ] & 0 \\ 0 & 0 & B_1^{2,2}(U)[\circ] \end{pmatrix},
    \end{align*}
    where 
    \begin{align*}
        \begin{cases}
            B_1^{1,1}(U)[\circ] = \epsilon \mathrm{G}_2 [ \mathrm{G}_2[v_1]\circ], \\
            B_1^{1,2}(U)[\circ] = (\mathrm{G}_1)^2[\circ] + \epsilon \mathrm{G}_2[\zeta \mathrm{G}_2[\circ]], \\
            B_1^{2,2}(U)[\circ] = \epsilon(\mathrm{G}_1)^2 [\mathrm{G}_2[v_1] \mathrm{G}_2[\circ]] + \epsilon^2 \mathrm{G}_2[\zeta\mathrm{G}_2[\mathrm{G}_2[v_1]\mathrm{G}_2[\circ]]],
        \end{cases}
    \end{align*}
    and 
    \begin{align*}
        B_2(U)[\circ] = \begin{pmatrix}  B_2^{1,1}(U)[\circ] & 0 & B_2^{1,3}(U)[\circ] \\ 0 & B_2^{2,2}(U)[\circ] & 0 \\ B_2^{1,3}(U)[\circ] & 0 & B_2^{2,2}(U)[\circ] \end{pmatrix},
    \end{align*}
    where 
    \begin{align*}
        \begin{cases}
            B_2^{1,1}(U)[\circ] = \epsilon \mathrm{G}_2 [\mathrm{G}_2[v_2]\circ], \\
            B_2^{1,3}(U)[\circ] = (\mathrm{G}_1)^2[\circ] + \epsilon \mathrm{G}_2[\zeta \mathrm{G}_2[\circ]], \\
            B_2^{2,2}(U)[\circ] = \epsilon(\mathrm{G}_1)^2 [\mathrm{G}_2[v_2] \mathrm{G}_2[\circ]] + \epsilon^2 \mathrm{G}_2[\zeta\mathrm{G}_2[\mathrm{G}_2[v_2]\mathrm{G}_2[\circ]]].
        \end{cases}
    \end{align*}
Finally, for $j = 1, 2$,
    \begin{align*}
        F_j(U)[\circ] = -\frac{B_j(U)[\circ]-B_j(U)^*[\circ]}{2\epsilon}.
    \end{align*}
\end{myppr}
\begin{proof}
Applying the matricial operator $S_0(U)[\circ] $ to
the system \eqref{WB matricial form introduction}, we immediately get
\begin{align*}
    S_0(U)[\partial_t U] + \sum\limits_{j=1}^d B_j(U)[\partial_j U] = 0.
\end{align*}
Then we write $ B_j(U)= \frac{B_j(U)+B_j(U)^*}{2}+ \frac{B_j(U)-B_j(U)^*}{2}=\widetilde{A}_j(U)-\epsilon F_j(U)$.
\end{proof}

\begin{mypp}\label{estimates remainder}
    Let $s \geq 0$ and $t_0 > d/2$. For any $\underline{U} \in X^{\max{(t_0+1,s)}}(\mathbb{R}^d)$ and any $U \in X^s(\mathbb{R}^d)$, for $j=1,2$, we have
    \begin{align}\label{estimates on F}
        |F_j(\underline{U})[\partial_j U]|_{Y^s} \leq C(|\underline{U}|_{X^{\max(t_0+1,s)}}) |U|_{X^s}.
    \end{align}
    It makes it a term of order $0$ with respect to the energy norm.
\end{mypp}
\begin{proof}
First remark that $(B_1^{1,2})^* = B_1^{1,2}$ and $(B_2^{1,3})^* = B_2^{1,3}$. So that we need to estimate the terms of $F_j(\underline{U})$, $j=1,2$, defined by $B_1^{1,1}$, $B_1^{2,2}$, $B_2^{1,1}$ and $B_2^{2,2}$. We shall only estimate the contributions from $B_2^{2,2}$, the other terms being similar.

 For the first contribution, we have
\begin{multline*}
    |(\mathrm{G}_1)^{-1}[(\mathrm{G}_1)^2 [\mathrm{G}_2[\underline{v_1}] \mathrm{G}_2[\partial_1 v_2]] - \mathrm{G}_2 [\mathrm{G}_2[\underline{v_1}] (\mathrm{G}_1)^2[\partial_1 v_2]]]|_{H^s} \\ \leq |\mathrm{G}_1[\mathrm{G}_2[\underline{v_1}]\mathrm{G}_2[\partial_1 v_2]] - \mathrm{G}_2[\mathrm{G}_2[\underline{v_1}]\mathrm{G}_1[\partial_1 v_2]]|_{H^s}
    + |(\mathrm{G}_1)^{-1}\mathrm{G}_2[[\mathrm{G}_1,\mathrm{G}_2[\underline{v_1}]]\mathrm{G}_1[\partial_1 v_2]]|_{H^s}  \colonequals I_1 + I_2.
\end{multline*}
We decompose
\begin{align*}
    I_1 = |[\mathrm{G}_1,\mathrm{G}_2[\underline{v_1}]]\mathrm{G}_2[\partial_1 v_2] - [\mathrm{G}_2,\mathrm{G}_2[\underline{v_1}]]\mathrm{G}_1[\partial_1 v_2]|_{H^s},
\end{align*}
so that using the commutator estimates of Proposition \ref{Commutator estimates} and the assumption $|G_2| \leq G_1\in L^\infty(\mathbb{R}^d)$ we get
\begin{align*}
    I_1 \lesssim |\mathrm{G}_2[\underline{v_1}]|_{H^{\max{(t_0+1,s)}}}|\mathrm{G}_1[\partial_1 v_2]|_{H^{s-1}} \lesssim |\underline{U}|_{X^{\max{(t_0+1,s)}}}|U|_{X^s}.
\end{align*}
Using the same tools, we get 
\begin{align*}
    I_2 \lesssim |\underline{U}|_{X^{\max{(t_0+1,s)}}}|U|_{X^s}.
\end{align*}

For the second contribution, using the same tools in addition to the product estimates of Proposition \ref{product estimate}, we have
\begin{align*}
    |(\mathrm{G}_1)^{-1}[\mathrm{G}_2[\underline{\zeta}\mathrm{G}_2[\mathrm{G}_2[\underline{v_1}]\mathrm{G_2}[\partial_1 v_2]]] &- \mathrm{G}_2[\mathrm{G}_2[\underline{v_1}]\mathrm{G}_2[\underline{\zeta}\mathrm{G}_2[\partial_1 v_2]]]]|_{H^s} \\
    \leq &|\underline{\zeta} \mathrm{G}_2[\mathrm{G}_2[\underline{v_1}]\mathrm{G_2}[\partial_1 v_2]] - \mathrm{G}_2[\underline{v_1}]\mathrm{G}_2[\underline{\zeta}\mathrm{G}_2[\partial_1 v_2]]|_{H^s} \\
    \leq &|-[\mathrm{G}_2,\underline{\zeta}](\mathrm{G}_2[\underline{v_1}]\mathrm{G}_2[\partial_1 v_2]) + [\mathrm{G}_2,\mathrm{G}_2[\underline{v_1}]](\underline{\zeta}\mathrm{G}_2[\partial_{1}v_2])|_{H^s} \\
    \leq &|\underline{\zeta}|_{H^{\max{(t_0+1,s)}}}|\mathrm{G}_2[\underline{v_1}]\mathrm{G}_2[\partial_1 v_2]|_{H^{s-1}} + |\mathrm{G}_2[\underline{v_1}]|_{H^{\max{(t_0+1,s)}}}|\underline{\zeta}\mathrm{G}_2[\partial_1 v_2]|_{H^{s-1}} \\
    \leq &|\underline{U}|_{X^{\max{(t_0+1,s)}}}^2|U|_{X^s}.
\end{align*}
As aforementioned, all other terms are estimated in the same way.
\end{proof}

\subsection{Energy estimates of order 0}
\label{Energy estimates of order 0 section}
In this subsection we prove the energy estimates of order 0 on the systems \eqref{WB matricial form introduction} linearized around a sufficiently regular state, specifically Proposition \ref{energy estimates introduction} with $s=0$.

\begin{mypp}\label{Energy estimates of order 0 proposition section}
	Let $t_0 > d/2$ and $(\mathrm{G}_1,\mathrm{G}_2)$ be admissible Fourier multipliers. For any $\epsilon\in(0,1]$, $T>0$, $\underline{U} = \big( \underline{\zeta} , \underline{v} \big) \in W^{1,\infty}([0,T/\epsilon],X^{t_0}(\mathbb{R}^d))\cap L^{\infty}([0,T/\epsilon],X^{t_0+1}(\mathbb{R}^d))$ for which there exists $h_{\min} > 0$ such that for all $(t,x) \in [0,T/\epsilon]\times \mathbb{R}^d$, \eqref{non cavitation} holds;
	and for any $U = \big( \zeta , v \big) \in W^{1,\infty}([0,T/\epsilon],X^{0}(\mathbb{R}^d))\cap L^{\infty}([0,T/\epsilon],X^{1}(\mathbb{R}^d))$ satisfying the system
	\begin{align*}
	\partial_t U + \sum\limits_{j=1}^d A_j(\underline{U})[\partial_j U] = \epsilon R,
	\end{align*}
	where $R \in L^{\infty}([0,T/\epsilon],X^0(\mathbb{R}^d))$, and for $j = 1,2$, $A_j(\underline{U})$ is defined by \eqref{A_j introduction}, we have for any $t\in [0,T/\epsilon]$,
\[
	|U|_{X^0} \leq \kappa_0 e^{\epsilon \lambda_0 t} |U|_{X^0}|_{t=0} + \epsilon \nu_0 \int_{0}^{t} |R(t')|_{X^0} \rm{d}t',
\]
	where $\lambda_0, \nu_0 := C(\frac{1}{h_{\min}},T,|\underline{U}|_{W^{1,\infty}_t X^{t_0}},|\underline{U}|_{L^{\infty}_t X^{t_0+1}})$ and $\kappa_0 \colonequals C(\frac{1}{h_{\min}},|\underline{U}|_{X^{t_0}}|_{t=0})$.
\end{mypp}

To prove this result we need some properties on the symmetrizer $S_0(\underline{U})$.
\begin{mylem}\label{coercivity symmetrizer}
    Let $t_0 > d/2$. Let $\underline{U} \in X^{t_0}(\mathbb{R}^d)$ be such that \eqref{non cavitation} is satisfied. The symmetrizer $S_0(\underline{U})$ satisfies for any $U \in X^0(\mathbb{R}^d)$
    \begin{align}\label{coercivity inequality}
        (S_0(\underline{U})U,U)_2 \geq |\zeta|^2_2 + h_{\min}|\mathrm{G}_1[v]|_2^2 \geq h_{\min} |U|_{X^0}^2.
    \end{align}
    and 
    \begin{align}\label{other estimates on the symmetrizer}
        \begin{cases}
            |S_0(\underline{U})[U]|_{Y^0} \leq C(|\underline{U}|_{X^{t_0}})|U|_{X^0}, \\ 
            (S_0(\underline{U})[U],U)_2 \leq C(|\underline{U}|_{X^{t_0}}) |U|_{X^0}^2.
        \end{cases}
    \end{align}
\end{mylem}
\begin{proof}
    Given the definition of $S_0(\underline{U})$ (see \eqref{symmetrizer}), the assumption $G_2 \leq G_1$ and the Sobolev embedding $H^{t_0} \subset L^{\infty}$, the estimates of \eqref{other estimates on the symmetrizer} are obvious. We prove here the inequality \eqref{coercivity inequality}.
    \begin{align*}
        (S_0(\underline{U})U,U)_2 &= |\zeta|_2^2 + ((\mathrm{G}_1)^2[v] + \mathrm{G}_2[\epsilon\underline{\zeta}\mathrm{G}_2[v]],v)_2\\
        & = |\zeta|^2_2 + (((\mathrm{G}_1)^2-(1-h_{\min})(\mathrm{G}_2)^2)[v], v)_2 + ((1-h_{\min})(\mathrm{G}_2)^2[v] + \mathrm{G}_2[\epsilon\underline{\zeta}\mathrm{G}_2[v]],v)_2
    \end{align*}
    Using $|G_2| \leq G_1$, $1-h_{\min}\geq 0$ and $\epsilon\underline{\zeta} \geq h_{\min}-1$ so:
    \begin{align*}
        (S_0(\underline{U})U,U)_2 &\geq |\zeta|_2^2 + h_{\min}|\mathrm{G}_1[v]|_2^2 +(1-h_{\min})|\mathrm{G}_2[v]|_2^2 + (h_{\min}-1)|\mathrm{G}_2[v]|_2^2  \\
        &\geq |\zeta|_2^2 + h_{\min}|\mathrm{G}_1[v]|_2^2,
    \end{align*}
    and the result is proved.
\end{proof}

\begin{mylem}\label{differential inequality}
    With the same assumptions as in Proposition \ref{Energy estimates of order 0 proposition section}, we have the following estimates 
    \begin{align*}
        \frac{\mathrm{d}}{\rm{d}t} (S_0(\underline{U})[U],U)_2 \leq  \epsilon C(|\underline{U}|_{W^{1,\infty}_t X^{t_0}},|\underline{U}|_{L^{\infty}_t X^{t_0+1}})(S_0(\underline{U})[U],U)_2 + \epsilon C(|\underline{U}|_{L^{\infty}_t X^{t_0}})|R|_{X^{0}}\sqrt{(S_0(\underline{U})[U],U)_2}.
    \end{align*}
\end{mylem}
\begin{proof}
Using the self-adjointness of $S_0(\underline{U})$, we have 
\begin{align*}
    \frac{1}{2} \frac{\mathrm{d}}{\rm{d}t}(S_0(\underline{U})[U],U)_2 &= \frac{1}{2} ((\partial_t S_0(\underline{U}))[U],U)_2 + (S_0(\underline{U})[\partial_t U],U)_2,
\end{align*}
where 
\begin{align*}
        \partial_t S_0(\underline{U})[\circ] = \begin{pmatrix} 0 & 0 \\ 0 & \epsilon\mathrm{G}_2[(\partial_t \underline{\zeta}) \mathrm{G}_2[\circ]]  \end{pmatrix}.
\end{align*}
Then using the symmetrization \eqref{symmetrized Wb equation general form}, we get
\begin{align*}
    \frac{1}{2} \frac{d}{\rm{d}t}(S_0(\underline{U})[U],U)_2 =  \frac{1}{2} ((\partial_t S_0(\underline{U}))[U],U)_2 - \sum \limits_{j=1}^{d} (\widetilde{A}_j(\underline{U})[\partial_j U],U)_2 + \epsilon (S_0(\underline{U})[R],U)_2 + \epsilon \sum\limits_{j=1}^d (F_j(U)[\partial_j U],U)_2.
\end{align*}
But using the symmetry of $\widetilde{A_j}$, $j=1,2$, we have
\begin{align*}
    (\widetilde{A}_j(\underline{U})[\partial_j U],U)_2 &= (\partial_j U, \widetilde{A}_j(\underline{U})[U])_2 \\
    &= -(U, (\partial_j \widetilde{A}_j(\underline{U}))[U])_2 - (U,\widetilde{A}_j(\underline{U})[\partial_j U])_2,
\end{align*}
where 
\begin{align*}
    (\partial_1 \widetilde{A}_1(\underline{U}))[\circ] = \frac{(\partial_1 B_1(\underline{U}))[\circ] + (\partial_1 (B_1(\underline{U})^*))[\circ]}{2},
\end{align*}
with 
\begin{align*} 
(\partial_1 B_1(\underline{U}))[\circ] = \begin{pmatrix} (\partial_1 B_1^{1,1}(\underline{U}))[\circ] & (\partial_1 B_1^{1,2}(\underline{U}))[\circ] & 0 \\ (\partial_1 B_1^{1,2}(\underline{U}))[\circ] & (\partial_1 B_1^{2,2}(\underline{U}))[\circ] & 0 \\ 0 & 0 & (\partial_1 B_1^{2,2}(\underline{U}))[\circ] \end{pmatrix},
\end{align*}
where 
\begin{align*}
    \begin{cases}
        (\partial_1 B_1^{1,1}(\underline{U}))[\circ] = \epsilon \mathrm{G}_2 [     \mathrm{G}_2[\partial_1 \underline{v_1}]\circ], \\
        (\partial_1 B_1^{1,2}(\underline{U}))[\circ] = \epsilon \mathrm{G}_2[(\partial_1 \underline{\zeta}) \mathrm{G}_2[\circ]], \\
        (\partial_1 B_1^{2,2}(\underline{U}))[\circ] = \epsilon(\mathrm{G}_1)^2 [\mathrm{G}_2[\partial_1 \underline{v_1}] \mathrm{G}_2[\circ]] + \epsilon^2 \mathrm{G}_2[(\partial_1\underline{\zeta})\mathrm{G}_2[\mathrm{G}_2[\underline{v_1}]\mathrm{G}_2[\circ]]] + \epsilon^2 \mathrm{G}_2[\underline{\zeta}\mathrm{G}_2[\mathrm{G}_2[\partial_1\underline{v_1}]\mathrm{G}_2[\circ]]],
    \end{cases}
\end{align*}
and 
\begin{align*}
    (\partial_2 \widetilde{A}_2(\underline{U}))[\circ] = \frac{(\partial_2 B_2(\underline{U}))[\circ] + (\partial_2 (B_2(\underline{U})^*))}{2},
\end{align*}
with 
\begin{align*} 
(\partial_2 B_2(\underline{U}))[\circ] = \begin{pmatrix} (\partial_2 B_2^{1,1}(\underline{U}))[\circ] & 0 &  (\partial_2 B_2^{1,3}(\underline{U}))[\circ] \\ 0 & (\partial_2 B_2^{2,2}(\underline{U}))[\circ] & 0 \\  (\partial_2 B_2^{1,3}(\underline{U}))[\circ] & 0 & (\partial_2 B_2^{2,2}(\underline{U}))[\circ] \end{pmatrix},
\end{align*}
where 
\begin{align*}
    \begin{cases}
        (\partial_2 B_2^{1,1}(\underline{U}))[\circ] = \epsilon \mathrm{G}_2 [     \mathrm{G}_2[\partial_2 \underline{v_2}]\circ], \\
        (\partial_2 B_2^{1,3}(\underline{U}))[\circ] = \epsilon \mathrm{G}_2[(\partial_2 \underline{\zeta}) \mathrm{G}_2[\circ]], \\
        (\partial_2 B_2^{2,2}(\underline{U}))[\circ] = \epsilon(\mathrm{G}_1)^2 [\mathrm{G}_2[\partial_2 \underline{v_2}] \mathrm{G}_2[\circ]] + \epsilon^2 \mathrm{G}_2[(\partial_2\underline{\zeta})\mathrm{G}_2[\mathrm{G}_2[\underline{v_2}]\mathrm{G}_2[\circ]]] + \epsilon^2 \mathrm{G}_2[\underline{\zeta}\mathrm{G}_2[\mathrm{G}_2[\partial_2\underline{v_2}]\mathrm{G}_2[\circ]]].
    \end{cases}
\end{align*}
For any $j = 1,2$, $(\partial_j (B_j(\underline{U})^*))[\circ]$ is also easily computed and have the same mathematical structure as $(\partial_j B_j(\underline{U}))[\circ]$.

So
\begin{align}\label{good equality energy estimates order 0}
    \frac{1}{2} \frac{\rm{d}}{\rm{d}t}(S_0(\underline{U})U,U)_2 = \frac{1}{2}((\partial_t S_0(\underline{U}) + \sum\limits_{j=1}^{d}\partial_j A_j(\underline{U}))[U],U)_2 + \epsilon (S_0(\underline{U})[R],U)_2 + \epsilon \sum\limits_{j=1}^d (F_j(U)[\partial_j U],U)_2.
\end{align}

The term $\epsilon(S_0(\underline{U})[R],U)_2$ is easily estimated by Cauchy-Schwarz inequality and \eqref{other estimates on the symmetrizer}:
\begin{align}\label{estimation on the remainder energy estimates of order 0 proof}
    \epsilon(S_0(\underline{U})[R],U)_2 \leq \epsilon |S_0(\underline{U})[R]|_{Y^0}|U|_{X^0} \leq \epsilon C(|\underline{U}|_{L^{\infty}_t X^{t_0}}) |R|_{X^0}|U|_{X^0}.
\end{align}

The term $\epsilon \sum\limits_{j=1}^d (F_j(U)[\partial_j U],U)_2$ is estimated using the Proposition \ref{estimates remainder}. Indeed, we have
\begin{align}\label{other term estimated}
    \epsilon \sum\limits_{j=1}^d (F_j(U)[\partial_j U],U)_2 \leq \epsilon \sum\limits_{j=1}^d |F_j(U)[\partial_j U]|_{Y^0}|U|_{X^0} \leq \epsilon C(|\underline{U}|_{L^{\infty}_t X^{t_0+1}})|U|_{X^{0}}^2.
\end{align}

Then, the result comes from the following estimates
\begin{align}\label{estimates derivated operators energy estimates order 0}
    |((\partial_t S_0(\underline{U}) + \sum\limits_{j=1}^{d} \partial_j A_j(\underline{U}))[U],U)_2| \leq \epsilon C(|\underline{U}|_{W^{1,\infty}_t X^{t_0}},|\underline{U}|_{L^{\infty}_t X^{t_0+1}})|U|_{X^0}^2.
\end{align}
We provide below some examples of controls needed to get the latter estimates:
\begin{align*}
    \begin{cases}
        |(\epsilon\mathrm{G}_2[(\partial_t \underline{\zeta}) \mathrm{G}_2[v], v)_2| \leq \epsilon |\partial_t\underline{\zeta}|_{L^{\infty}_t H^{t_0}}|\mathrm{G}_1[v]|_2^2 \leq \epsilon |\partial_t\underline{\zeta}|_{L^{\infty}_t H^{t_0}} |U|_{X^0}^2,\\
        |(\epsilon \mathrm{G}_2 [(\partial_1\underline{\zeta})\mathrm{G}_2[\zeta]],v_1)_2| \leq \epsilon |\underline{\zeta}|_{L^{\infty}_t H^{t_0+1}} |\zeta|_2|\mathrm{G}_1[v_1]|_2 \leq \epsilon |\underline{\zeta}|_{L^{\infty}_t H^{t_0+1}} |U|_{X^0}^2, \\
        |(\epsilon^2 \mathrm{G}_2[\partial_1\underline{\zeta}\mathrm{G}_2[\mathrm{G}_2[\underline{v_1}]\mathrm{G}_2[v_1]]], v_1)_2| \leq \epsilon^2 |\underline{\zeta}|_{L^{\infty}_t H^{t_0+1}}|\mathrm{G}_2[\mathrm{G}_2[\underline{v_1}]\mathrm{G}_2[v_1]]|_2|\mathrm{G}_2[v_1]|_2 \\
       \qquad  \leq \epsilon^2 |\underline{\zeta}|_{L^{\infty}_t H^{t_0+1}}|\mathrm{G}_1[\underline{v_1}]|_{L^{\infty}_t H^{t_0}}|\mathrm{G}_1[v_1]|_2^2 \leq \epsilon^2 |\underline{\zeta}|_{L^{\infty}_t H^{t_0+1}}|\underline{v_1}|_{L^{\infty}_t X^{t_0}} |U|_{X^0}^2,
    \end{cases}
\end{align*}
where we used the assumption $|G_2| \leq G_1$ and the boundedness of the latter Fourier multiplier in $H^s(\mathbb{R}^d)$ (see Proposition \ref{trivial}).
All other contributions are estimated using the same tools.

Combining \eqref{good equality energy estimates order 0}, \eqref{estimation on the remainder energy estimates of order 0 proof}, \eqref{other term estimated}, \eqref{estimates derivated operators energy estimates order 0} and \eqref{coercivity inequality} in Lemma \ref{coercivity symmetrizer} we get the desired differential inequality.
\end{proof}

We now have all the elements needed to prove the Proposition \ref{Energy estimates of order 0 proposition section}. 
\begin{proof}[Proof of Proposition \ref{Energy estimates of order 0 proposition section}. ]
By Lemma \ref{differential inequality}, we have
\begin{multline*}
      \sqrt{(S_0(\underline{U})[U],U)_2} \frac{\rm{d}}{\rm{d}t} \sqrt{(S_0(\underline{U})[U],U)_2} =  \frac{1}{2}\frac{\rm{d}}{\rm{d}t} (S_0(\underline{U})[U],U)_2 \\
     \leq  \epsilon C(\frac{1}{h_{\min}},|\underline{U}|_{W^{1,\infty}_t X^{t_0}},|\underline{U}|_{L^{\infty}_t X^{t_0+1}})(S_0(\underline{U})[U],U)_2 + \epsilon C(\frac{1}{h_{\min}},|\underline{U}|_{L^{\infty}_t X^{t_0}})|R|_{X^{0}}\sqrt{(S_0(\underline{U})[U],U)_2}.
\end{multline*}
 Dividing by $\sqrt{(S_0(\underline{U})[U],U)_2}$, we get
\begin{align*}
     \frac{\rm{d}}{\rm{d}t} \sqrt{(S_0(\underline{U})[U],U)_2} \leq \epsilon \lambda_0 \sqrt{(S_0(\underline{U})[U],U)_2} + \epsilon C(\frac{1}{h_{\min}},|\underline{U}|_{L^{\infty}_t X^{t_0}})|R|_{X^{0}},	
\end{align*} 
 with $\lambda_0= C(\tfrac{1}{h_{\min}},|\underline{U}|_{W^{1,\infty}_t X^{t_0}},|\underline{U}|_{L^{\infty}_t X^{t_0+1}})$.
We can integrate this inequality in time between $0$ and $t$ to get
\begin{align*}
     \sqrt{(S_0(\underline{U})[U],U)_2}  \leq e^{\epsilon \lambda_0 t} \sqrt{(S_0(\underline{U})[U],U)_2}|_{t=0} + \epsilon C(\tfrac{1}{h_{\min}},|\underline{U}|_{L^{\infty}_t X^{t_0}}) \int_{0}^{t} e^{\epsilon (t-t')\lambda_0}|R(t')|_{X^0}^2 \rm{d}t',
\end{align*} 
And using Lemma \ref{coercivity symmetrizer} yields the desired estimate.
\end{proof}
\subsection{Energy estimates of higher order}
\label{Energy estimates of higher order section}
We now prove Proposition \ref{energy estimates introduction} which we recall here for the sake of clarity.
\begin{mypp}\label{Energy estimates of order s proposition section}
  Let $s \geq 0$, $t_0 > d/2$ and $(\mathrm{G}_1,\mathrm{G}_2)$ be admissible Fourier multipliers. For any $\epsilon\in(0,1]$, $T>0$, $\underline{U} = \big( \underline{\zeta} , \underline{v} \big) \in W^{1,\infty}([0,T/\epsilon],X^{t_0}(\mathbb{R}^d))\cap L^{\infty}([0,T/\epsilon],X^{\max{(t_0+1,s)}}(\mathbb{R}^d))$ for which there exists $h_{\min} > 0$ such that for all $(t,x) \in [0,T/\epsilon]\times \mathbb{R}^d$ , \eqref{non cavitation} holds;
and for any $U = \big( \zeta , v \big) \in W^{1,\infty}([0,T/\epsilon],X^{s}(\mathbb{R}^d))\cap L^{\infty}([0,T/\epsilon],X^{s+1}(\mathbb{R}^d))$ satisfying the system
\begin{align*}
\partial_t U + \sum\limits_{j=1}^d A_j(\underline{U})[\partial_j U] = \epsilon R,
\end{align*}
where $R \in L^{\infty}([0,T/\epsilon],X^s(\mathbb{R}^d))$, and for $j = 1,2$, $A_j(\underline{U})$ is defined by \eqref{A_j introduction}, we have for any $t\in [0,T/\epsilon]$,
\begin{align}\label{energy estimates of order s}
|U|_{X^s} \leq \kappa_0 e^{\epsilon \lambda_s t} |U|_{X^s}|_{t=0} + \epsilon \nu_s \int_{0}^{t} |R(t')|_{X^s} \rm{dt}',
\end{align}
where $\lambda_s, \nu_s := C(\frac{1}{h_{\min}},T,|\underline{U}|_{W^{1,\infty}_t X^{t_0}},|\underline{U}|_{L^{\infty}_t X^{\max{(t_0+1,s)}}})$ and $\kappa_0 \colonequals C(\frac{1}{h_{\min}},|\underline{U}|_{X^{t_0}}|_{t=0})$.
\end{mypp}

We begin by proving the following lemma.
\begin{mylem}\label{Lemma energy estimates of order s WB}
    With the same assumptions as Proposition \ref{Energy estimates of order s proposition section}, there exists $R_{(s)} \in L^{\infty}([0,T/\epsilon],X^0(\mathbb{R}^d))$ such that
    \begin{align*}
     \partial_t \Lambda^s U + \sum\limits_{j=1}^{d} A_j(\underline{U}) \partial_j \Lambda^s U = \epsilon R_{(s)},
    \end{align*}
    with $\Lambda^s \colonequals (1-\Delta)^{s/2}$ and
    \begin{align*}
        |R_{(s)}|_{X^0} \leq C(\tfrac{1}{h_{\min}},|\underline{U}|_{L^{\infty}_t X^{\max{(t_0+1,s)}}})(|U|_{X^s}+|R|_{X^s}).
    \end{align*}
\end{mylem}

\begin{proof}
We know that 
\begin{align*}
    &\partial_t U + \sum\limits_{j=1}^{d} A_j(\underline{U})\partial_j U = \epsilon R.
\end{align*}
Applying the operator $\Lambda^s = (1-\Delta)^{s/2}$ to this equation, we get
\begin{align*}
    \partial_t \Lambda^s U + \sum\limits_{j=1}^{d} A_j(\underline{U})[\partial_j \Lambda^s U] = \epsilon \Lambda^s R - \sum\limits_{j=1}^{d} [\Lambda^s,A_j(\underline{U})][\partial_j U].
\end{align*}
But because $\Lambda^s$ commutes with $G_1$ and $G_2$, we have
\begin{align}\label{A_j lambdas commutator}
    \begin{cases} [\Lambda^s,A_1(\underline{U})][\circ] = \begin{pmatrix} \epsilon \mathrm{G}_2 [ [\Lambda^s,\mathrm{G}_2[\underline{v_1}]][\circ]] & \epsilon \mathrm{G}_2[[\Lambda^s,\underline{\zeta}] \mathrm{G}_2[\circ]] & 0 \\ 0 & \epsilon [\Lambda^s,\mathrm{G}_2[\underline{v_1}]] \mathrm{G}_2[\circ] & 0 \\ 0 & 0 & \epsilon [\Lambda^s,\mathrm{G}_2[\underline{v_1}]] \mathrm{G}_2[\circ] \end{pmatrix}, \\ 
    [\Lambda^s,A_2(\underline{U})][\circ] = \begin{pmatrix}  \epsilon \mathrm{G}_2 [[\Lambda^s,\mathrm{G}_2[\underline{v_2}]]\circ] & 0 & \epsilon\mathrm{G_2}[[\Lambda^s,\underline{\zeta}]\mathrm{G}_2[\circ]] \\ 0 & \epsilon[\Lambda^s,\mathrm{G}_2[\underline{v_2}]] \mathrm{G}_2[\circ] & 0 \\ 0 & 0 & \epsilon[\Lambda^s,\mathrm{G}_2[\underline{v_2}]] \mathrm{G}_2[\circ] \end{pmatrix}.
    \end{cases}
\end{align}
So that using the commutator estimates of Proposition \ref{Commutator estimates standard} and $|G_2|\leq G_1\in L^\infty(\mathbb{R}^d)$, we get
\begin{align*}
    |\sum\limits_{j=1}^{d} [\Lambda^s,A_j(\underline{U})][\partial_j U]|_{X^0} \leq \epsilon C(|\underline{U}|_{L^{\infty}_t X^{\max{(t_0+1,s)}}})|U|_{X^s}.
\end{align*}
At the end, denoting $R_{(s)} = \Lambda^s R - \frac{1}{\epsilon}\sum\limits_{j=1}^{d} [\Lambda^s,A_j(\underline{U})][\partial_j U]$, we get the result.

\end{proof}

We now prove Proposition \ref{Energy estimates of order s proposition section}.
\begin{proof}[Proof of Proposition \ref{Energy estimates of order s proposition section}.]
Using the energy estimates and notations of Proposition \ref{Energy estimates of order 0 proposition section} and Lemma \ref{Lemma energy estimates of order s WB} we get
\begin{multline*}
|U|_{X^s} \leq \kappa_0 e^{\epsilon\lambda_0 t}|U|_{X^s}|_{t=0} + \epsilon \nu_0 \int_0^t C(\tfrac{1}{h_{\min}},|\underline{U}|_{L^{\infty}_t X^{\max{(t_0+1,s)}}})  |U|_{X^s} {\rm d}t' \\
        + \epsilon \nu_0 \int_0^t C(\tfrac{1}{h_{\min}},|\underline{U}|_{L^{\infty}_t X^{\max{(t_0+1,s)}}}) |R(t')|_{X^s} {\rm d}t',
\end{multline*}
and using Gronwall's lemma, we get
\[
    |U|_{X^s} \leq (\kappa_0 e^{\epsilon\lambda_0 t}|U|_{X^s}|_{t=0}+ \epsilon \nu_0 C_s \int_0^t  |R(t')]|_{X^s} {\rm d} t')e^{\epsilon \nu_0 C_s t} 
    \]
    with $C_s=C(\tfrac{1}{h_{\min}},|\underline{U}|_{L^{\infty}_t X^{\max{(t_0+1,s)}}})$, and hence
\[ |U|_{X^s} \leq \kappa_0 e^{\epsilon\lambda_s t}|U|_{X^s}|_{t=0} + \epsilon \nu_s \int_0^t |R(t')|_{X^s} {\rm d}t',
\]
where $\nu_s, \lambda_s := C(\frac{1}{h_{\min}}, T,|\underline{U}|_{W^{1,\infty}_t X^{t_0}},|\underline{U}|_{L^{\infty}_t X^{\max{(t_0+1,s)}}})$. This concludes the proof.
\end{proof}

\section{Local well-posedness and stability}
\label{Local well-posedness and stability section}
In this section we prove Theorem \ref{Thm wanted}, following the regularization technique employed for instance in the Chapter 7 of \cite{Metivier2008} to symmetrizable quasi-linear hyperbolic systems of conservation laws.
\subsection{Well-posedness of the linearized systems}
\label{Well-posedness of the linearized systems section}
In this subsection we study the local well-posedness of the systems \eqref{WB matricial form introduction} linearized around a sufficiently regular state.
\begin{mythm}\label{well posedness linearized system}
    Let $s > d/2 +1$, $h_{\min} >0$ and $0 \leq \epsilon \leq 1$. Let also $(\mathrm{G}_1,\mathrm{G}_2)$ be a couple of admissible Fourier multipliers. Let $\underline{U} \in W^{1,\infty}([0,T/\epsilon],X^{s-1}(\mathbb{R}^d)) \cap L^{\infty}([0,T/\epsilon],X^{s}(\mathbb{R}^d))$  be such that \eqref{non cavitation} is satisfied. Let $R \in L^{\infty}([0,T/\epsilon],X^s(\mathbb{R}^d))$ and $U_0 \in X^s(\mathbb{R}^d)$. The Cauchy problem 
    \begin{align}\label{system linearized around underline U}
        \begin{cases}
            \partial_t U + \sum\limits_{j=1}^d A_j(\underline{U})[\partial_j U] = \epsilon R,\\
            U|_{t=0} = U_0,
        \end{cases}
    \end{align}
    where $A_j(\underline{U})$ is defined by \eqref{A_j introduction}, has a unique solution in $C^0([0,T/\epsilon],X^s(\mathbb{R}^d))$. Moreover, the solution satisfies the energy estimates \eqref{energy estimates of order s}.
\end{mythm}

As said previously, to prove this result, we use a regularization method. 

Let $J_{\alpha} = (1-\alpha\Delta)^{-1/2}$, $\alpha \in ]0,1]$, be a regularizing Fourier multiplier. We have the following properties:
\begin{itemize}
    \item For $\alpha > 0$, $J_{\alpha}$ is a regularizing operator of order -1.
    \item For any $s \geq 0$, the family $\{ J_{\alpha}, \alpha \in ]0,1] \}$ is uniformly bounded in $X^s(\mathbb{R}^d)$.
    \item For any $s \geq 0$, and for all $v \in X^s(\mathbb{R}^d), J_{\alpha}v \to v$ in $X^s(\mathbb{R}^d)$ as $\alpha \to 0$.
\end{itemize}

We decompose the proof of Theorem \ref{well posedness linearized system} into several lemmas. 
\begin{mylem}
    With the same assumptions as in Theorem \ref{well posedness linearized system}, the Cauchy problem \eqref{system linearized around underline U} has a weak solution $U \in L^{\infty}([0,T/\epsilon],X^s(\mathbb{R}^d))$.
\end{mylem}
\begin{proof}
    We consider the Cauchy problem
    \begin{align}\label{Cauchy problem proof existence theorem}
        \begin{cases}
            \partial_t U_{\alpha} + \sum\limits_{j=1}^d A_j(\underline{U})\partial_j J_{\alpha}U_{\alpha} = \epsilon R,\\
            U_{\alpha}|_{t=0} = U_0.
        \end{cases}
    \end{align}
    For $j=1,2$ the operator $A_j(\underline{U})\partial_j J_{\alpha}$ is bounded in $X^s(\mathbb{R}^d)$, so that the Cauchy-Lipschitz theorem gives the existence of a solution $U_{\alpha} \in C^0([0,T/\epsilon],X^s(\mathbb{R}^d))$. Because $J_{\alpha}$ is a Fourier multiplier of order $0$ uniformly in $\alpha$, and is bounded in $X^s(\mathbb{R}^d)$ uniformly in $\alpha$, we can get the same energy estimates for \eqref{Cauchy problem proof existence theorem} as the ones for \eqref{system linearized around underline U}. It implies that the sequence $U_{\alpha}$ is bounded in $L^{\infty}([0,T/\epsilon],X^s(\mathbb{R}^d))$.
    
    Now recall that $L^{\infty}([0,T/\epsilon],X^s(\mathbb{R}^d))$ is the dual of $L^1([0,T/\epsilon],Y^{-s}(\mathbb{R}^d))$. So by the weak* compactness of the closed balls of the dual of a normed space, there exists a subsequence, still denoted $U_{\alpha}$, which converges weak* as $\alpha \to 0$ to an element $U$. Passing to the limit in the sense of distributions in \eqref{Cauchy problem proof existence theorem} we get $U \in L^{\infty}([0,T/\epsilon],X^s(\mathbb{R}^d))$ is a weak solution of \eqref{system linearized around underline U}.
    
    It still remains to prove that $U|_{t=0} = U_0$ makes sense (and holds). Remark that from the equation $\partial_t U \in L^{\infty}([0,T/\epsilon],X^{s-1}(\mathbb{R}^d))$. Hence $U \in C^0([0,T/\epsilon],X^{s-1}(\mathbb{R}^d))$, and it makes sense to take the trace at $t=0$ of $U$ in $X^{s-1}(\mathbb{R}^d)$, and we do have $U|_{t=0} = U_0$ from the limiting process.
\end{proof}

\begin{mylem}\label{Uniqueness}
    Let $R \in L^{\infty}([0,T/\epsilon],X^s(\mathbb{R}^d))$ and suppose that $U \in L^{\infty}([0,T/\epsilon],X^s(\mathbb{R}^d))$ satisfies the linearized system \eqref{system linearized around underline U} with $U_0 = U|_{t=0} \in X^s(\mathbb{R}^d)$. Then $U \in C^0([0,T/\epsilon],X^s(\mathbb{R}^d))$ and satisfies the energy estimates \eqref{energy estimates of order s}. 
\end{mylem}
\begin{proof}
    Applying $J_{\alpha}$ to the system \eqref{system linearized around underline U}, we get
    \begin{align}\label{regularized linearized system}
        \partial_t J_{\alpha}U + \sum\limits_{j=1}^d A_j(\underline{U})[\partial_j J_{\alpha}U] = \epsilon J_{\alpha}R - \sum\limits_{j=1}^d [J_{\alpha},A_j(\underline{U})][\partial_j U].
    \end{align}
    We denote $R_{(\alpha)} = J_{\alpha} - \frac{1}{\epsilon}\sum\limits_{j=1}^d [J_{\alpha},A_j(\underline{U})][\partial_j U]$. We easily see that $R_{(\alpha)} \in L^{\infty}([0,T/\epsilon],X^s(\mathbb{R}^d))$ using the same argument as for the proof of Lemma \ref{Lemma energy estimates of order s WB} and the fact that $J_{\alpha}$ is of order $0$.
    
    Moreover from the density of $X^{s+1}(\mathbb{R}^d)$ in $X^s(\mathbb{R}^d)$, we get
    \begin{align*}
        [J_{\alpha},A_j(\underline{U})][\partial_j U] \to 0 \ \ \mathrm{in} \ \ L^{\infty}([0,T/\epsilon],X^s(\mathbb{R}^d)),
    \end{align*}
    as $\alpha \to 0$. It implies that $R_{(\alpha)} \to R$ in $L^{\infty}([0,T/\epsilon],X^s(\mathbb{R}^d))$.
    
    We know that $J_{\alpha}U$ is in $L^{\infty}([0,T/\epsilon],X^{s+1}(\mathbb{R}^d))$, so using the equation \eqref{regularized linearized system}, we get $\partial_t J_{\alpha} U \in L^{\infty}([0,T/\epsilon],X^{s}(\mathbb{R}^d))$. The Sobolev embedding in dimension 1 gives $J_{\alpha}U \in C^0([0,T/\epsilon],X^{s}(\mathbb{R}^d))$.
    Using the energy estimates of order $s$ \eqref{energy estimates of order s} on $J_{\alpha}U - J_{\alpha'}U$ we get that $(J_{\alpha}U)_{\alpha \geq 0}$ is a Cauchy sequence in $C^0([0,T/\epsilon],X^s(\mathbb{R}^d))$ as $\alpha \to 0$. So $J_{\alpha}U$ converges in $C^0([0,T/\epsilon],X^s(\mathbb{R}^d))$. But $J_{\alpha}U$ converges to $U$ in $L^{\infty}([0,T/\epsilon],X^s(\mathbb{R}^d))$. Thus $J_{\alpha}U \to U$ in $C^0([0,T/\epsilon],X^s(\mathbb{R}^d))$ as $\alpha \to 0$.
    
    Using again the energy estimates of order $s$ \eqref{energy estimates of order s} but this time on $J_{\alpha}U$ and passing to the limit $\alpha \to 0$, we get that $U$ satisfies the energy estimates of order $s$.
\end{proof}

We now complete the proof of Theorem \ref{well posedness linearized system}.
\begin{proof}
    The two previous lemmas provide the existence of a solution $U \in C^0([0,T/\epsilon],X^s(\mathbb{R}^d))$ which satisfies the energy estimates \eqref{energy estimates of order s}. 
    
    It only remains to prove the uniqueness. For two solutions $U_1$ and $U_2$ with the same initial condition $U_0 \in X^s(\mathbb{R}^d)$, the difference $V = U_1 - U_2$ satisfies the system
    \begin{align*}
        \begin{cases}
            \partial_t V + \sum\limits_{j=1}^d A_j(\underline{U}) \partial_j V = 0, \\
            V|_{t=0} = 0.
        \end{cases}
    \end{align*}
    There remains to use the energy estimates of Proposition \ref{Energy estimates of order 0 proposition section} to infer $V = 0$.
\end{proof}

\subsection{Well-posedness of the non-linear systems}
\label{Well-posedness of the systems section}
This subsection is dedicated to the proof of Theorem \ref{Thm wanted} which we recall here for the sake of clarity.

\begin{mythm}\label{Theorem Well posedness WB}
    Let $s > d/2 +1$, $h_{\min} >0$ and $M > 0$. Let also $(\mathrm{G}_1,\mathrm{G}_2)$ be a couple of admissible Fourier multipliers. There exist $T >0$ and $C > 0$ such that for all $\epsilon \in (0,1]$, $U_0 \in X^s(\mathbb{R}^d)$ with $|U_0|_{X^s} \leq M$ and satisfying \eqref{non cavitation}, there exists a unique solution $U \in C^0([0,T/\epsilon],X^s(\mathbb{R}^d))$ of the Cauchy problem
    \begin{align}\label{Cauchy problem}
        \begin{cases}
            \partial_t U + \sum\limits_{j=1}^d A_j(U)\partial_j U = 0, \\
            U|_{t=0} = U_0.
        \end{cases}
    \end{align}
    Moreover $|U|_{L^{\infty}([0,T/\epsilon], X^s)} \leq C |U_0|_{X^s}$.
\end{mythm}

\begin{proof}
Consider the iterative scheme $U_0(t,x) = U_0(x)$ and for $n \in \mathbb{N}$
\begin{align}\label{iterative scheme}
    \begin{cases}
        \partial_t U_{n+1} + \sum\limits_{j=1}^d A_j(U_n)\partial_j U_{n+1} = 0, \\
        U_{n+1}|_{t=0} = U_0.
    \end{cases}
\end{align}.

\begin{mylem}\label{Uniform bounds}
    There exists $T>0$ as in Theorem \ref{Theorem Well posedness WB} such that the sequences $U_n$ and $\partial_t U_n$ are well defined and are bounded in respectively $C^0([0,T/\epsilon],X^s(\mathbb{R}^d))$ and $C^0([0,T/\epsilon],X^{s-1}(\mathbb{R}^d))$.
\end{mylem}
\begin{proof}
    We prove by induction that there exists $C_1,C_2 > 0$ and $T> 0$ as in Theorem \ref{Theorem Well posedness WB} for all $n \in \mathbb{N}$
    \begin{align*}
        \sup\limits_{t \in [0,T/\epsilon]} |U_n|_{X^{s}} \leq C_1|U_0|_{X^s}, \ \ \sup\limits_{t\in[0,T/\epsilon]} |\partial_t U_n|_{X^{s-1}} \leq C_2|U_0|_{X^s}, \ \ \inf\limits_{t \in [0,T/\epsilon], X \in \mathbb{R}^d}(1 + \epsilon \zeta_{n}(t,x)) \geq h_{\min}/2.
    \end{align*}
    By Theorem \ref{well posedness linearized system} $U_{n+1}$ is well-defined and satisfies \eqref{energy estimates of order s}. Specifically, on the time interval $[0,T/\epsilon]$ we have
    \begin{align*}
        |U_{n+1}|_{X^s} \leq \kappa_0 e^{\epsilon \lambda_n t}|U_{0}|_{X^s}.
    \end{align*}
    where $\lambda_n = C(T,|U_n|_{W^{1,\infty}_t X^{t_0}},|U_n|_{L^{\infty}_t X^{\max{(t_0+1,s)}}})$ and $\kappa_0 = C(|U_0|_{X^{t_0}}|_{t=0})$. 
    
    \noindent Also, using the equation and the product estimates of Proposition \ref{product estimate}, 
    \begin{align*}
        |\partial_t U_{n+1}|_{X^{s-1}} \leq \widetilde{C}(|U_n|_{X^s})|U_{n+1}|_{X^s}.
    \end{align*}
    Moreover
    \begin{align*}
        \epsilon\zeta_{n+1}(t,x) = \epsilon\zeta_0(x) + \epsilon\int_0^t \partial_t \zeta_{n+1}(T,x) \rm{dt}'.
    \end{align*}
    But from the Sobolev embedding there exists $C_s > 0$ such that
    \begin{align*}
        |\partial_t \zeta_{n+1}(t,x)| \leq C_s |\partial_t \zeta|_{L^{\infty}_t X^{s-1}}.
    \end{align*}
    So
    \begin{align*}
        1+ \epsilon \zeta_{n+1} \geq h_{\min} - T C_s |\partial_t \zeta|_{L^{\infty}_t X^{s-1}}.
    \end{align*}
    Let $C_1 > \kappa_0$. Let $C_2$ be such that $\widetilde{C}(C_1|U_0|_{X^s})C_1 \leq C_2$. And let $T$ be sufficiently small so that, $\kappa_0 e^{\lambda T} \leq C_1$ where $\lambda_n \leq \lambda = C(T,C_1|U_0|_{X^s},C_2|U_0|_{X^s})$, and $T C_s C_2 |U_0|_{X^s} \leq h_{\min}/2$.
\end{proof}

\begin{mylem}\label{Cauchy sequence}
    The sequence $U_n$ is a Cauchy sequence in $C^0([0,T/\epsilon],X^0(\mathbb{R}^d))$.
\end{mylem}
\begin{proof}
    Let $V_n \colonequals U_{n+1} - U_n$. For $n\geq1$, it satisfies
    \begin{align*}
        \begin{cases}
            \partial_t V_n + \sum\limits_{j=1}^d A_j(U_n) \partial_j V_n = \epsilon R_n,\\
            V_{n+1}|_{t=0} = 0,
        \end{cases}
    \end{align*}
    where 
    \begin{align*}
        R_n = -\frac{1}{\epsilon} \sum\limits_{j=1}^d (A_j(U_n)-A_j(U_{n-1})) \partial_j U_n.
    \end{align*}
    But from the expression of $A_j$ \eqref{A_j introduction} and the uniform bounds of the sequence $U_n$ (see Lemma \ref{Uniform bounds}) and the product estimates of Proposition \ref{product estimate} it is easy to see that there exists a constant $M > 0$ independent of $n$ such that
    \begin{align*}
        |R_n|_{X^0} \leq  M |V_{n-1}|_{X^0}.
    \end{align*}
    And from the energy estimates \eqref{energy estimates of order s} and the uniform bounds of $U_n$ and $\partial_t U_n$ (see again Lemma \ref{Uniform bounds}), there exists a constant $M > 0$ independent of $n$ such that 
    \begin{align*}
        |V_n|_{X^0} \leq \epsilon M \int_0^t |V_{n-1}(t')|_{X^0} \rm{dt}'.
    \end{align*}
    So
    \begin{align*}
        |V_n|_{X^0} \leq \frac{M^n t^n}{n!} \sup\limits_{t \in [0,T/\epsilon]} |V_0|_{X^0}.
    \end{align*}
    Thus, the series $\sum V_n$ converges in $C^0([0,T/\epsilon],X^0(\mathbb{R}^d))$.
\end{proof}

    We now complete the proof of Theorem \ref{Theorem Well posedness WB}.
    
    From Lemma \ref{Cauchy sequence}, the sequence $U_n$ converges in $C^0([0,T/\epsilon],X^0(\mathbb{R}^d))$. From Lemma \ref{Uniform bounds}, the sequence $U_n$ is uniformly bounded in $C^0([0,T/\epsilon],X^s(\mathbb{R}^d))$. So for any $s' < s$, $U_n$ converges in $C^0([0,T/\epsilon],X^{s'}(\mathbb{R}^d))$. Take $s' > t_0 + 1$ and denote by $U$ the limit. The sequences $U_n$, $\partial_t U_n$ and for $j=1,2$, $\partial_j U_n$ converge uniformly in $C^0$ to respectively $U$, $\partial_t U$ and $\partial_j U$. Hence $U$ is solution to \eqref{WB matricial form introduction}. Moreover, from Lemma \ref{Uniform bounds}, $U \in L^{\infty}([0,T/\epsilon],X^s(\mathbb{R}^d))$, $\partial_t U \in L^{\infty}([0,T/\epsilon],X^{s-1}(\mathbb{R}^d))$, and $U$ satisfies the estimates of Theorem \ref{Theorem Well posedness WB}. So we can consider $U$ as a solution of the linearized system \eqref{system linearized around underline U} taking $\underline{U}$ as $U$. The Theorem \ref{well posedness linearized system} gives $U \in C^0([0,T/\epsilon],X^s(\mathbb{R}^d))$ and its uniqueness as a solution of the Cauchy problem \eqref{Cauchy problem}. This concludes the proof of Theorem \ref{Theorem Well posedness WB}.
\end{proof}

\subsection{Blow up criterion}
\label{Blow up criterion section}
From Theorem \ref{Theorem Well posedness WB}, one can define the maximal time existence $T^* > 0$ of the solution $U \in C^0([0,T^*/\epsilon),X^s(\mathbb{R}^d))$ of the Cauchy problem \eqref{Cauchy problem} associated to an initial condition $U_0 \in X^s(\mathbb{R}^d)$ such that \eqref{non cavitation} holds.  

\begin{mypp}\label{blow up proposition}
    We have 
    \begin{align*}
        T^* < +\infty \implies
        \lim\limits_{t \to T^*/\epsilon} |U|_{L^{\infty}([0,t],X^s)} = +\infty.
    \end{align*}
\end{mypp}
\begin{proof}
    Suppose for the sake of contradiction that $T_\star=+\infty$ and there exists $M > 0$ such that 
    \begin{align*}
        |U|_{L^{\infty}([0,T^*/\epsilon),X^s)} = M.
    \end{align*}
    Then from Theorem \ref{Theorem Well posedness WB}, there exists $T>0$ such that for any $\beta >0$, and $T_{\beta} = T^*-\beta$, the Cauchy problem \eqref{Cauchy problem} with initial condition $U(T_{\beta})$ has a unique solution in $C^0([T_{\beta}/\epsilon,T_1/\epsilon],X^s(\mathbb{R}^d))$ with $T_1 = T_{\beta} + T > T^*$. Taking $\beta = T/2$, by uniqueness, $U$ has an extension $\widetilde{U} \in C^0([0,T_1/\epsilon],X^s(\mathbb{R}^d))$ solution of the Cauchy problem \eqref{Cauchy problem}. Thus, necessarily, $T^* = +\infty$. If not, it contradicts the definition of $T^*$. 
\end{proof}

\section{Stability}
\label{Stability result section}
In this section we prove the stability result of Proposition \ref{Stability}. 

\begin{mypp}\label{Stability pas intro}
    Let the assumptions of Theorem \ref{Theorem Well posedness WB} be satisfied and use the notations therein. Assume also that there exists $\widetilde{U} \in C([0,0,\widetilde T/\epsilon],X^s(\mathbb{R}^d))$ solution of 
    \begin{align*}
    \partial_t\widetilde{U} + \sum\limits_{j=1}^d A_j(\widetilde{U})\partial_j\widetilde{U} = \widetilde{R},
    \end{align*}
    where $\widetilde{R} \in L^{\infty}([0,T/\epsilon],X^{s-1}(\mathbb{R}^d))$. Then, the error with respect to the solution $U \in C^0([0,T/\epsilon],X^s(\mathbb{R}^d))$ given by Theorem \ref{Theorem Well posedness WB} satisfies for all times $t \in [0,\min{(\widetilde T,T)}/\epsilon]$,
    \begin{align*}
    |\mathfrak{e}|_{L^{\infty}([0,t],X^{s-1})} \leq C(\tfrac{1}{h_{\min}},|U|_{L^{\infty}([0,t], X^s)},|\widetilde{U}|_{L^{\infty}([0,t], X^s)})(|\mathfrak{e}|_{X^{s-1}}|_{t=0} + t|\widetilde{R}|_{L^{\infty}([0,t], X^{s-1})}),
    \end{align*}
    where $\mathfrak{e} \coloneqq U - \widetilde{U}$.
\end{mypp}

\begin{proof}
    We know that 
    \begin{align}\label{both systems}
        \begin{cases}
            \partial_t U + \sum\limits_{j=1}^d A_j(U)\partial_j U = 0,\\
            \partial_t\widetilde{U} + \sum\limits_{j=1}^d A_j(\widetilde{U})\partial_j\widetilde{U} = \widetilde{R}.
        \end{cases}
    \end{align}
    Subtracting both equations we get
    \begin{align}\label{error equation}
        \partial_t \mathfrak{e} + \sum\limits_{j=1}^d A_j(U)\partial_j\mathfrak{e} = \epsilon F,
    \end{align}
    where 
    \begin{align*}
        F = -\frac{1}{\epsilon}\widetilde{R} - \frac{1}{\epsilon}\sum\limits_{j=1}^d(A_j(U)-A_j(\widetilde{U}))\partial_j\widetilde{U}.
    \end{align*}
    We can easily estimate $F$ using the product estimates of Proposition \ref{product estimate} ($s-1 > d/2$):
    \begin{align*}
        |F|_{X^{s-1}} \leq \frac{1}{\epsilon}|\widetilde{R}|_{X^{s-1}} + |\widetilde{U}|_{X^{s}}|\mathfrak{e}|_{X^{s-1}}.
    \end{align*}
    We use the energy estimates of Proposition \ref{Energy estimates of order s proposition section} on \eqref{error equation} to get
    \begin{align*}
        |\mathfrak{e}|_{X^{s-1}} \leq \kappa_0 e^{\epsilon\lambda_{s-1}t}|\mathfrak{e}|_{X^{s-1}}|_{t=0} + \nu_{s-1} \int_{0}^t |\widetilde{R}(t')|_{X^{s-1}} \mathrm{d}t' + \epsilon \nu_{s-1} \int_0^t |\widetilde{U}|_{X^s}|\mathfrak{e}|_{X^{s-1}}\mathrm{d}t'.
    \end{align*}
    Using Gronwall's lemma, we then have
    \begin{align*}
        |\mathfrak{e}|_{X^{s-1}} \leq (\kappa_0 e^{\epsilon\lambda_{s-1}t}|\mathfrak{e}|_{X^{s-1}}|_{t=0} + \nu_{s-1} t |\widetilde{R}(t')|_{L^{\infty}([0,t], X^{s-1})} \mathrm{d}t')e^{\epsilon \nu_{s-1} |\widetilde{U}|_{L^{\infty}([0,t], X^s)}t}.
    \end{align*}
    It only remains to see that using the equation on $U$ of \eqref{both systems} and the product estimates of Proposition \ref{product estimate}, we have for all times $t \in [0,\min(\widetilde T,T)/\epsilon]$,
    \begin{align*}
        |\partial_t U|_{X^{s-1}} \leq C(|U|_{X^s}), 
    \end{align*}
    to get the result.
\end{proof}

\section{Full justification of a class of Whitham-Boussinesq systems}
\label{Full justification of a class of Whitham-Boussinesq systems section}
In this section we prove the full justification of a class of Whitham-Boussinesq systems, Theorem \ref{justification}, recalled below.

\begin{mythm}
    	Under the assumption and using the notation of Proposition \ref{consistency}, and provided  $(\mathrm{G}_1,\mathrm{G}_2)$ are admissible Fourier multipliers, then for any $U = (\zeta,\nabla\psi) \in C^0([0,\widetilde T/\epsilon],H^{s+4}(\mathbb{R}^d))$ classical solution of the water waves equations and satisfying the non-cavitation assumption \eqref{non cavitation}, there exists a unique $U_{\mathrm{WB}} = (\zeta_{\mathrm{WB}},v_{\mathrm{WB}}) \in C^0([0,T/\epsilon],X^{s+4}(\mathbb{R}^d))$ classical solution of the Whitham-Boussinesq systems \eqref{WB matricial form introduction} with initial data $U_{\mathrm{WB}}|_{t=0} = (\zeta|_{t=0},\nabla\psi|_{t=0})$, and one has for all times $t \in [0,\min{(\widetilde T,T)}/\epsilon]$,
    \begin{align*}
    |U - U_{\mathrm{WB}}|_{L^{\infty}([0,t],X^{s})} \leq C\, \mu\epsilon t ,
    \end{align*}
    with $T$ (provided by Theorem \ref{Theorem Well posedness WB}) and  $C=C(\tfrac{1}{h_{\min}},|U|_{L^{\infty}([0,T/\epsilon], H^{s+4})})$ uniform with respect to $(\mu,\epsilon)\in (0,1]^2$.
\end{mythm}
\begin{proof}
    It is important to remark that, as pointed out in Remark \ref{rem}, the dependency with our previous results with respect to admissible pairs of Fourier multipliers $(\mathrm{G}_1,\mathrm{G}_2)$ occurs only through Proposition \ref{trivial} and Proposition \ref{Commutator estimates} (in addition to $|G_2|\leq G_1$), and hence through the quantity $|(G_k,\langle \cdot \rangle \nabla G_k)|_{L^\infty}$ for $k\in\{1,2\}$.  Considering Fourier multipliers of the form $\mathrm{G}_1^{\mu}=G_1(\sqrt\mu D)$ and $\mathrm{G}_2^{\mu}=G_2(\sqrt\mu D)$ (see Notation \ref{notation G dependant de mu}), we can remark that the above quantity is non-increasing as $\mu$ decreases, and hence all estimates proved in this paper hold uniformly with respect to $\mu\in(0,1]$. In particular, the existence time in Theorem \ref{Theorem Well posedness WB} and the energy estimates of Proposition \ref{Energy estimates of order s proposition section} are independent of $\mu\in(0,1]$. 
    
    Now, from the continuous embedding $H^{s'}(\mathbb{R}^d) \subset X^{s'}(\mathbb{R}^d)$ (for any $s'\in \mathbb{R}$) and Theorem \ref{Theorem Well posedness WB} we have the existence and uniqueness of $U_{\mathrm{WB}} \in C^0([0,T/\epsilon],X^{s+4}(\mathbb{R}^d))$  with the control   
    	\[ |U_{\mathrm{WB}}|_{L^{\infty}([0,T/\epsilon],X^{s+4})} \lesssim |(\zeta|_{t=0},\nabla\psi|_{t=0})|_{X^{s+4}}\leq |(\zeta|_{t=0},\nabla\psi|_{t=0})|_{H^{s+4}} ,\]
    	 with $T$ independent of $\mu$. From Proposition \ref{consistency}, we know that $U$ satisfies
    \begin{align*}
        \partial_t U + \sum\limits_{j=1}^d A_j(U)\partial_j U = \mu\epsilon R,
    \end{align*}
    where, for any $t\in[0,\widetilde T/\epsilon]$, $|R(t,\cdot)|_{H^s} \leq C(\frac{1}{h_{\min}}, |\zeta|_{H^{s+4}},|\nabla\psi|_{H^{s+4}})$. 
    From the stability result of Proposition \ref{Stability pas intro}, we infer that for all times $t \in [0,\min{(T,\widetilde T')}/\epsilon]$, one has
    \begin{align*}
        |U - U_{\mathrm{WB}}|_{L^{\infty}([0,t],X^{s})} \leq \mu\epsilon t C(\tfrac{1}{h_{\min}},|U|_{L^{\infty}([0,t], X^{s+1})},|U_{\mathrm{WB}}|_{L^{\infty}([0,t], X^{s+1})})|R|_{L^{\infty}([0,t], X^{s})}.
    \end{align*}
    The result follows from combining the previous estimates and using once again the continuous embedding $H^{s'}(\mathbb{R}^d) \subset X^{s'}(\mathbb{R}^d)$ for $s'=s$ and $s'=s+1$.
\end{proof}

\appendix
\section{Technical tools}
\label{annex}
\begin{mypp}[Product estimates]\label{product estimate}
    Let $t_0 > d/2$, $s\geq-t_0$ and $f \in H^s\cap H^{t_0}(\mathbb{R}^d), g\in H^s(\mathbb{R}^d)$. Then $fg \in H^s(\mathbb{R}^d)$ and 
    \begin{align*}
        |fg|_{H^s} \leq C\, |f|_{H^{\max{(t_0,s)}}}|g|_{H^s}
    \end{align*}
    with $C$ depending uniquely on $s$ and $t_0$.
\end{mypp}
\begin{proof}
See Proposition B.2 in \cite{WWP}.
\end{proof}
\begin{mypp}[Commutator estimates with symbols of order $s$]\label{Commutator estimates standard}
	Let $t_0 > d/2$, $s\geq0$, and denote $\Lambda^s=(\Id-\Delta)^{s/2}$. Then for any $f\in H^{s}\cap H^{t_0+1}(\mathbb{R}^d)$ and for all $g \in H^{s-1}(\mathbb{R}^d)$,
	\begin{align*}
	|[\Lambda^s,f]g|_{L^2} \leq C\, |f|_{H^{\max{(t_0+1,s)}}}|g|_{H^{s-1}}
	\end{align*}
	with $C$ depending uniquely on $s$ and $t_0$.
\end{mypp}
\begin{proof}
	See Corollary B.9 in \cite{WWP}.
\end{proof}

\begin{mypp}\label{trivial}
	Let $G \in L^{\infty}(\mathbb{R}^d)$. Then for any $s\in\mathbb{R}$ and $f\in H^{s}(\mathbb{R}^d)$, then $G(D)f \in H^{s}(\mathbb{R}^d)$ and
	\begin{align*}
	|G(D)f|_{H^s} \leq |G|_{L^\infty}|f|_{H^s}.
	\end{align*}
\end{mypp}
\begin{proof}
	The result is immediate by Parseval's theorem.
\end{proof}
\begin{mypp}[Commutator estimates with symbols of order $0$]\label{Commutator estimates}
	Let $t_0 > d/2$, $s \geq 0$ and $G \in W^{1,\infty}(\mathbb{R}^d)$ be such that $\langle \cdot \rangle \nabla G \in L^\infty(\mathbb{R}^d)$ and for any $f\in H^{s}\cap H^{t_0+1}(\mathbb{R}^d)$ then, for all $g \in H^{s-1}(\mathbb{R}^d)$,
	\begin{align*}
	|[G(D),f]g|_{H^s} \leq C\, |f|_{H^{\max{(t_0+1,s)}}}|g|_{H^{s-1}}
	\end{align*}
	    with $C$ depending uniquely on $s$ and $t_0$, and $|(G,\langle \cdot \rangle \nabla G)|_{L^\infty}$.
\end{mypp}
\begin{proof}
	See Lemma 2.5 in \cite{DucheneMelinand22}.
\end{proof}

\bibliographystyle{plain}
\bibliography{Biblio.bib}

\end{document}